\documentclass[oneside,english]{amsart}

\usepackage{amssymb}
\usepackage{amscd}

\numberwithin{equation}{section} 
\numberwithin{figure}{section} 
\theoremstyle{plain}
\newtheorem*{thm*}{Theorem}
\theoremstyle{plain}
\newtheorem{thm}{Theorem}[section]
\theoremstyle{definition}
\newtheorem{defn}[thm]{Definition}
\theoremstyle{plain}
\newtheorem{lem}[thm]{Lemma}
\theoremstyle{plain}
\newtheorem{prop}[thm]{Proposition}
\theoremstyle{plain}

\theoremstyle{remark}

\theoremstyle{remark}
\newtheorem*{acknowledgement*}{Acknowledgement}

\begin{document}

\title{Frobenius integrability and Finsler metrizability for
  $2$-dimensional sprays}

\author[Bucataru]{Ioan Bucataru}
\address{Faculty of  Mathematics \\ Alexandru Ioan Cuza University \\ Ia\c si, 
  Romania}
\email{bucataru@uaic.ro}
\urladdr{http://www.math.uaic.ro/\textasciitilde{}bucataru/}
\author[Cre\c tu]{Georgeta Cre\c tu}
\address{Faculty of  Mathematics \\ Alexandru Ioan Cuza University \\ Ia\c si, 
  Romania}
\email{Georgeta.Cretu@math.uaic.ro}
\author[Taha]{Ebtsam H. Taha}
\address{Department of Mathematics, Faculty of Science,
Cairo University, 12613 Giza, Egypt}
\email{ebtsam@sci.cu.edu.eg, ebtsam.h.taha@hotmail.com}

\date{\today}

\begin{abstract}
  For a $2$-dimensional non-flat spray we associate a Berwald frame and a
  $3$-dimensional distribution that we call the Berwald distribution. The Frobenius integrability of the Berwald
  distribution characterises the Finsler metrizability of the given spray. In the integrable case, the
  sought after Finsler function is provided by a closed, homogeneous $1$-form from the
  annihilator of the Berwald distribution. We discuss both the
  degenerate and non-degenerate cases using the fact that the regularity of the Finsler
  function is encoded into a regularity condition of a
  $2$-form, canonically associated to the given spray. The integrability of the Berwald distribution and the
  regularity of the $2$-form have simple and useful expressions in terms of the Berwald frame.  
\end{abstract}



\maketitle

\section{Introduction}
\label{Introduction}

The inverse problem of Lagrangian mechanics requires to decide whether
or not a given system of second order ordinary differential equations (SODE or
semispray) can be derived from a variational principle. In general, the problem is far from being solved. It has been
completely solved in dimension $1$ by Darboux \cite{Darboux94} and
in dimension $2$ by Douglas \cite{Douglas41}.

In the case when the given SODE is homogeneous (i.e, a spray), the
problem is known as the Finsler metrizability problem
\cite{BM14, GM00, KS85, SV02}. A solution to this problem has been proposed in the analytic case by
Muzsnay \cite{Muzsnay06}, by studying the formal integrability of
the associated Euler-Lagrange PDE. 

In this paper we provide a constructive solution to the $2$-dimensional case of the Finsler
metrizability problem in the non-flat case. We associate to a given non-flat
$2$-dimensional spray two canonical
geometric structures that contain all the information about the
Finsler metrizability of the spray and, in the affirmative case, they
give the Finsler function. One of these structures is a $3$-dimensional
regular distribution, called the \emph{Berwald distribution},
whose integrability provides a candidate for the Finsler function that
we look for. The second structure is a $2$-form, whose rank gives the
information about the regularity of the Finsler candidate. A key
aspect in our approach is the use of a canonical frame,
called the \emph{Berwald frame}, associated to the given spray that
makes easier to express the integrability of the Berwald distribution
as well as the rank of the $2$-form. Therefore, we reformulate and
solve the Finsler metrizability problem in terms of some properties of the
Berwald frame. In the integrable case, the Berwald distribution
coincides with the holonomy distribution \cite{Muzsnay06} and the number of solutions
we provide agrees with the metrizability freedom of a spray introduced in \cite{EM16}.

The Berwald frame has been introduced first, locally, for a
$2$-dimensional Finsler metric in \cite{Berwald41}. An intrinsic
formulation of the Berwald frame, in the Finslerian setting, has been
provided in \cite{VV01}, see also \cite[\S 9.9.1]{SLK14}. Such a frame has been rediscovered recently
for a background Riemannian metric in \cite{Crampin14} to give an
alternative proof of the projective Finsler metrizability property of
an arbitrary $2$-dimensional spray. In our case, we define the Berwald
frame directly for an arbitrary spray and use its properties to obtain
information about the Finsler metrizability of the given spray.   

For the geodesic spray of a Finsler function, it is known that one can
always construct an integrable distribution, transverse to the Liouville vector field,
that is tangent to the indicatrix of the Finsler function
\cite{BF06}. In dimension $2$, our Theorems \ref{thm1} and \ref{thm2} provide a
characterization of the Finsler metrizability in terms of the
integrability of such a distribution. This distribution is the Berwald distribution
and in the integrable case it is tangent to the indicatrix of the
Finsler function that metrizes the given spray.

\section{Sprays and Finsler functions, a geometric setting}
\label{Sprays}

For a $2$-dimensional smooth, orientable and real manifold $M$, we denote by
$(TM, \pi, M)$ its tangent bundle and by $T_0M:=TM\setminus \{0\}$ the
total space of the tangent bundle with the zero section removed. Local coordinates $(x^i)$ on $M$ induce
local coordinates $(x^i, y^i)$ on $TM$ and $T_0M$. There are some
canonical structures that naturally live on $TM$. The vertical
distribution, $u\in TM \mapsto V_uTM:=\operatorname{Ker}(\pi_{\ast})_u
\subset T_uTM$ is a regular, integrable, $2$-dimensional distribution, and locally generated
by $(\partial/\partial y^1, \partial/\partial y^2)$. The Liouville
(dilation) vector field ${\mathcal C}=y^i\partial/\partial y^i$ is a
vertical vector field whose one-parameter group of diffeomorphisms is
generated by the positive homotheties of the fibres. The tangent structure
(vertical endomorphism) is given by $J=\partial/\partial y^i\otimes
dx^i.$

\subsection{A geometric setting for sprays} 
A spray (i.e., a homogeneous SODE) on $M$ is a vector field $S\in {\mathfrak
  X}(T_0M)$ that satisfies $JS={\mathcal C}$ (is a second order vector
field) and $[{\mathcal C}, S]=S$ (it is a positively $2$-homogeneous vector
field, and we will refer to this by saying that it is $2^{+}$-homogeneous). Locally, a spray can be expressed as follows:
\begin{eqnarray*}
S=y^i\frac{\partial}{\partial x^i} - 2G^i\frac{\partial}{\partial y^i}
  \label{spray}, \end{eqnarray*} 
where the functions $G^i$ are locally defined, and $2^{+}$-homogeneous in
the fiber coordinates.

To a given spray $S$, one can associate, using the Fr\"olicher-Nijenhuis formalism, a geometric setting including a
nonlinear connection, dynamical covariant derivative and curvature
tensors \cite{BD09, Grifone72a, GM00, Szilasi03}.

A spray $S$ induces a \emph{nonlinear connection} through an endomorphism
$\Gamma = [J, S]$ on $T_0M$ \cite{Grifone72a}. The connection $\Gamma$
is an \emph{almost product structure}, which means that
$\Gamma^2=\operatorname{Id}$, and induces two projectors 
\begin{eqnarray}
h=\frac{1}{2}\left(\operatorname{Id}+[J,S]\right), \quad
  v=\frac{1}{2}\left(\operatorname{Id} - [J,S]\right). \label{hv}
\end{eqnarray}
The projector $v$ corresponds to the vertical distribution $VTM$, while
$h$ induces a horizontal distribution $HTM$, which is supplementary to
the vertical distribution. Locally, the two projectors can be
expressed as follows:
\begin{eqnarray*}
h=\frac{\delta}{\delta x^i} \otimes dx^i, \ v=\frac{\partial}{\partial
  y^i}\otimes \delta y^i, \ \frac{\delta}{\delta
  x^i}:=\frac{\partial}{\partial x^i} - N^j_i \frac{\partial}{\partial
  y^j}, \ \delta y^i:=dy^i + N^i_jdx^j,\ N^i_j:=\frac{\partial
  G^i}{\partial y^j}. \label{hvlocal} 
\end{eqnarray*} 

The connection induced by a spray can induces an
\emph{almost complex structure} on $T_0M$ given by 
\begin{align}
{\mathbb F}=h\circ [S, h] - J = \frac{\delta}{\delta x^i} \otimes
  \delta y^i - \frac{\partial}{\partial y^i} \otimes dx^i, \label{acs}
\end{align}
see, e.g. \cite[(3.14)]{GM00}.

The \emph{dynamical covariant derivative} $\nabla$, induced by a spray $S$, is a tensor derivation on
$T_0M$, whose action on functions and vector fields is given by,
\cite[\S 3.2]{BD09}, 
\begin{eqnarray}
\nabla f=S(f), \ \nabla X=h[S, hX]+v[S, vX], \  f\in
  C^{\infty}(T_0M), \  X\in {\mathfrak X}(T_0M). \label{nabla}
\end{eqnarray} 
The \emph{Jacobi endomorphism}, of a spray $S$ is the vector valued $1$-form $\Phi=v\circ [S,
h]$, whose local expression is 
\begin{eqnarray*}
\Phi = R^i_j\frac{\partial}{\partial y^i}\otimes dx^j, \
  R^i_j=2\frac{\partial G^i}{\partial x^j} -
  S\left(\frac{\partial G^i}{\partial y^j}\right) - \frac{\partial
  G^i}{\partial y^k}\frac{\partial G^k}{\partial y^j}. \label{Rij} \end{eqnarray*}
A spray $S$ is called \emph{isotropic} if there exist a function $\rho
\in C^{\infty}(T_0M)$ and a semi-basic $1$-form
$\alpha=\alpha_idx^i\in \Lambda^1(T_0M)$ such that 
\begin{eqnarray}
\Phi=\rho J - \alpha \otimes {\mathcal C} \Longleftrightarrow
  R^i_j=\rho \delta^i_j - \alpha_j y^i. \label{isotropic} \end{eqnarray}
The function $\rho=\operatorname{Tr}(\Phi)=R^1_1+R^2_2$, which is called the \emph{Ricci scalar}, is
related to the semi-basic $1$-form $\alpha$ by $\rho=i_S\alpha
= \alpha_1y^1+\alpha_2y^2$. 

The homogeneity property of a spray $S$ is inherited by all associated
geometric structures. The Jacobi endomorphism is $2^{+}$-homogeneous and
therefore the Ricci scalar $\rho$ is $2^{+}$-homogeneous, while the
semi-basic $1$-form $\alpha$ is $1^{+}$-homogeneous.

It is known that $2$-dimensional sprays
are always isotropic, see \cite[Lemma 8.1.10]{Shen01} or
\cite[Corollary 8.3.11]{SLK14}, which means that their Jacobi endomorphism is
given by formula \eqref{isotropic}, with   
\begin{eqnarray}
\alpha_1=\frac{R^2_2}{y^1}=\frac{-R^2_1}{y^2}, \
  \alpha_2=\frac{R^1_1}{y^2}=\frac{-R^1_2}{y^1}. \label{alpha12} \end{eqnarray}   

In this work we will pay attention to non-flat sprays. This assumption
means that the Jacobi endomorphism $\Phi$ is nowhere vanishing, which is
equivalent to $\alpha\neq 0$ and it implies that $\rho\neq 0$. 

We will use the geometric setting described above to address the
following metrizability problem for a given spray $S$. Is there a
Finsler function whose geodesic spray is $S$? We will provide the answer
to this problem, using two geometric structures. The first structure is a
distribution, whose integrability will provide the metric candidate. The second
structure is a $2$-form that encodes information about the regularity
of the metric.     

For a spray $S$, the non-flatness assumption implies that the
distribution 
\begin{align}
{\mathcal D}=\operatorname{Im}(h)\oplus
  \operatorname{Im}(\Phi) \label{dhphi} \end{align}
is a regular, $3$-dimensional distribution. We call ${\mathcal D}$ the \emph{Berwald distribution} of the spray
$S$ and, we will prove that the spray is
metrizable if and only if ${\mathcal D}$ is integrable. 

Another distribution that can be associated to an arbitrary spray $S$
has been introduced by Muzsnay in \cite{Muzsnay06}. It is called the
\emph{holonomy distribution} ${\mathcal D}_{\mathcal H}$,
generated by horizontal vector fields and their succesive Lie
brackets. The holonomy distribution has been recently used in
\cite{EM16} to discuss the metrizability freedom of a spray. The two
distributions are related by ${\mathcal D}\subset {\mathcal
  D}_{\mathcal H}$.  

For an isotropic spray $S$ with Jacobi
endomorphism $\Phi$ given by formula \eqref{isotropic}, we consider the following
$2$-form: 
\begin{align}
\Omega=d\left(\frac{\alpha}{\rho}\right) + 2  i_{{\mathbb F}}\frac{\alpha}{\rho}
  \wedge \frac{\alpha}{\rho}. \label{omegas}
\end{align} 
In the case of metrizability, the rank of the $2$-form $\Omega$ will
provide information about the regularity of the Finsler function
which we search for. The idea of considering the $2$-form $\Omega$ has the origins
in the ``Scalar Flag Curvature''-test proposed in \cite[Theorem 3.1]{BM14}. 

\subsection{A geometric setting for Finsler functions}

For the metrizability problem of a given spray we pay attention to the
non-Riemannian case, by searching for Finslerian solutions. 

\begin{defn} \label{Finsler_def}
A continuous function $F:TM \to {\mathbb R}$ is called a \emph{Finsler
  function} if it satisfies the following conditions:
\begin{itemize}
\item[i)] $F$ is smooth and strictly positive on $T_0M$;
\item[ii)] $F$ is positively $1^{+}$-homogeneous in the fiber coordinates,
  which means $F(x,\lambda y)=\lambda F(x,y)$, $\forall \lambda \geq
  0$ and $(x,y)\in TM$;
\item[iii)] $F^2$ is a regular Lagrangian, which means that the
  following \emph{metric tensor} has maximal rank $2$ on $T_0M$
\begin{eqnarray*} g_{ij}=\frac{1}{2}\frac{\partial^2
  F^2}{\partial y^i\partial y^j}. \end{eqnarray*}
\end{itemize}
\end{defn}
If in the definition we restrict the domain of
the function $F$ to some open cone $A\subset TM$, then we speak about a
\emph{conic pseudo-Finsler function}. Also, if we replace the
regularity condition iii) above by the \emph{weak regularity condition}
$\operatorname{rank}(g_{ij})=1$, then we speak about a
\emph{degenerate Finsler function}.
 
The regularity conditions can be reformulated in terms of the
Hilbert 2-form $\omega_{F^2}:=-dd_JF^2$ as follows. The function $F$
satisfies iii), and hence it is a Finsler function, if and
only if $\omega_{F^2}$ has maximal rank $4$, which means that $\omega_{F^2}$ is a symplectic
structure. The function $F$ satisfies the weak regularity condition, and hence it
is a degenerate Finsler function if and only if the $2$-form
$\omega_{F^2}$ has rank $2$. 

A spray $S$ is \emph{Finsler metrizable} if there exists a (degenerate, conic
pseudo-) Finsler function that satisfies the Euler-Lagrange equation
\begin{align} i_Sdd_JF^2=-dF^2. \label{isomega} \end{align}
Equation \eqref{isomega} expresses the fact that the base integral curves of the
spray (homogeneous SODE) are solutions to the Euler-Lagrange equations
for $F^2$. In other words, $S$ is the geodesic spray, when $F$ is a
Finsler function. In the degenerate case, $S$ is one of the geodesic
sprays of $F$.

Consider a geodesic spray $S$ of a (degenerate) Finsler function $F$
and let $\Gamma$ be the connection induced by $S$. The Hilbert $2$-form
of $F$ can be expressed in terms of the connection $\Gamma$ as
follows:
\begin{align} \omega_{F^2}=2g_{ij} dx^i\wedge \delta
  y^j. \label{ogij} \end{align}
Within this geometric setting, we can associate to a  (degenerate)
Finsler function $F$ a (semi-) Riemannian metric on $T_0M$,
\cite{Grifone72a},  by
\begin{align} & 2G(X, Y) = -\omega_{F^2}(X, {\mathbb F}Y), \quad
  \omega_{F^2}(X,Y)=-2G(X,JY) + 2G(JX,Y), \label{gomega} \\ 
&  G = g_{ij}dx^i\otimes dx^j + g_{ij}\delta y^i\otimes \delta y^j. \nonumber
\end{align}
If $S$ is the geodesic spray of a Finsler function then we use the
metric $G$ above to construct a distribution that is orthogonal to the
Liouville vector field. The integrability of this distribution has
been shown in \cite{BF06}, the leaves of the foliation determined by
this distribution being given by the indicatrix of the Finsler
function.

The main results of our work provide a converse of this result. If for
a spray $S$, the Berwald distribution is integrable, then it will be
tangent to the indicatrix of a Finsler function that metrizes the
given spray. 

Consider a (degenerate, conic pseudo-) Finsler function $F$, a
spray $S$ satisfying equation \eqref{isomega}, and its Jacobi
endomorphism $\Phi$. We say that $S$ has \emph{scalar flag curvature} $\kappa \in
C^{\infty}(T_0M)$ if the Jacobi endomorphism $\Phi$ is of the form
\begin{eqnarray}
\Phi=\kappa F^2 J - \kappa F d_JF \otimes {\mathcal C}. \label{sfc}
\end{eqnarray}
When $\kappa$ is a constant, we say that $S$ has \emph{constant flag
  curvature}. The non-flatness assumption that we work with is
equivalent to $\kappa\neq 0$. 

The Euler-Lagrange equation \eqref{isomega}, satisfied by the geodesic spray $S$ of a
(degenerate) Finsler function $F$ is equivalent to $d_hF^2=0$. This
implies 
\begin{align*}
dF^2=d_vF^2=i_{{\mathbb F}}d_JF^2. \end{align*}
It is known that if the geodesic spray of a Finsler function is
isotropic then the Finsler function has scalar flag curvature, see
\cite[Lemma 8.2.2]{Shen01}. We will give here an alternative proof of
this result, for the more general case that includes the degenerate
Finsler functions.
\begin{prop} \label{prop1}
Consider an isotropic spray $S$ that is also a geodesic spray of a (degenerate) Finsler function $F$. Then the spray $S$ has scalar flag curvature.
\end{prop}
\begin{proof}
Consider the Euler-Lagrange equation \eqref{isomega}, which is
equivalent to $d_hF^2=0$. As a consequence we obtain $S(F^2)=0$. We
also have
\begin{align*}
d_{[S,h]}F^2={\mathcal L}_S d_hF^2 - d_h{\mathcal L}_S F^2=0.
\end{align*}
Using this equation, the fact that $dF^2=d_vF^2$ and the 
definition of the Jacobi endomorphism, we obtain
\begin{align*}
0=d_{[S,h]}F^2 =i_{[S,h]} dF^2 = i_{[S,h]} d_vF^2 = i_{[S,h]} i_v dF^2
  = i_{v\circ [S,h]} dF^2 = i_{\Phi}dF^2 = d_{\Phi} F^2. 
\end{align*}
We use now the assumption that $S$ is isotropic, which means that the
Jacobi endomorphism is given by formula \eqref{isotropic}. We replace
this in the above formula to obtain 
\begin{align*}
0= d_{\Phi} F^2 = d_{\rho J - \alpha \otimes {\mathcal C}}F^2  =\rho
  d_JF^2 - \alpha {\mathcal L}_{{\mathcal C}}F^2 = \rho
  d_JF^2 - 2F^2 \alpha, \end{align*}
whence
\begin{align}
\frac{\alpha}{\rho} = \frac{d_JF^2}{2F^2}=\frac{d_JF}{F}. \label{ar}\end{align}
Using the function $\kappa=\rho/F^2$ and the formulae
\eqref{ar}, the Jacobi endomorphism can be written as follows: 
\begin{align*}
\Phi = \rho\left(J - \frac{\alpha}{\rho} \otimes {\mathcal C}\right) =
  \frac{\rho}{F^2} \left( F^2 J - F^2 \frac{d_JF}{F} \otimes {\mathcal C}\right) = \kappa \left( F^2 J - Fd_JF \otimes {\mathcal C} \right).
\end{align*} 
This shows that formula \eqref{sfc} is true, and hence the spray $S$
has scalar flag curvature $\kappa$.  
\end{proof}
For an isotropic spray $S$ that is also a geodesic spray of a
(degenerate) Finsler function $F$, in view of formula \eqref{ar},  we
obtain that the $2$-form $\Omega$, \eqref{omegas}, is related to the Hilbert $2$-form $\omega_{F^2}$
by
\begin{align}
\Omega= d\left(\frac{1}{2F^2}d_JF^2\right) + \frac{1}{2F^4} dF^2
  \wedge d_JF^2 = \frac{1}{2F^2}dd_JF^2 = \frac{-1}{2F^2}\omega_{F^2}. \label{oomega}
\end{align}

Since $2$-dimensional sprays are always isotropic, it follows using
Proposition \ref{prop1} that Finsler metrizable $2$-sprays have
scalar flag curvature.

\section{Berwald frame.}
\label{Berwald}

In \cite{Berwald41} Berwald constructed a frame on $T_0M$
canonically associated to a $2$-dimensional Finsler manifold and used to characterize projectively flat $2$-dimensional
Finsler manifold and to classify them in some particular cases:
Landsberg spaces and Finsler spaces with the main scalar a function of
position only. A detailed analysis of the role played by the Berwald frame for the geometry of a
$2$-dimensional Finsler space is presented in \cite[Section
3.5]{AIM93}. Recently, in \cite{Crampin14}, Crampin rediscovered such a 
frame, which he associated to a given Riemannian metric, to give
a new constructive proof of the fact that any $2$-dimensional spray is 
projectively Finsler metrizable. 

In this section we will show that such a frame, which we will call the
\emph{Berwald frame}, can be associated to any $2$-dimensional 
non-flat spray. In the next section we will prove that a spray is
metrizable if and only if the Berwald distribution is integrable and
the $2$-form $\Omega$ satisfies some regularity
conditions. Both the integrability and the regularity conditions can
be expressed in terms of the Berwald
frame. 

In the following subsections we study the regularity conditions and
the commutation formulae satisfied by the Berwald frame in three cases: for a spray in
general, for a Finsler metrizable spray, and for a spray metrizable
by a degenerate Finsler function. 

\subsection{Berwald frame for a spray.} \label{Bs}
For a $2$-dimensional spray $S$, we consider the geometric setting
described in the previous section. As we already mentioned, $S$ is
isotropic and therefore its Jacobi endomorphism is given by formula
\eqref{isotropic}. We make the assumption that $S$ is non-flat and
therefore the semi-basic $1$-form $\alpha$ and the Ricci scalar $\rho$ are nowhere vanishing on $T_0M$. If we allow to work with conic pseudo-Finsler
functions, then we will restrict the domain to some open cone $A\subset
T_0M$, where $\alpha$ and $\rho$ are not vanishing.

Consider a vector field $H\in {\mathfrak X}(T_0M)$ that satisfies the
following conditions:
\begin{align}
[{\mathcal C}, H]=H, \quad h(H)=H, \quad \alpha(H)=0. \label{Berwald_s1} 
\end{align}  
First two conditions \eqref{Berwald_s1} express that $H$ is a $2^{+}$-homogeneous
horizontal vector field. Last condition above is equivalent to
$\Phi(H)=\rho JH$ and means that $H$ is (fibrewise) an eigenvector for the Jacobi
endomorphism $\Phi$ corresponding to the non-vanishing eigenvalue $\rho$. 

Conditions \eqref{Berwald_s1} do not determine uniquely the vector
field $H$, such vector field is determined only up to a $0^{+}$-homogeneous 
function factor. We can fix such a vector field $H$ by requiring that $\{S, H\}$ is
compatible with a fixed orientation of $M$. Since
$\alpha(S)=\rho\neq 0$ and $\alpha(H)=0$ we obtain that $H$ and $S$ are
two linearly independent vector fields that generate the horizontal
distribution. It follows that $V:=JH$ and ${\mathcal C}=JS$ are two
linearly independent vector fields that generate the vertical
distribution. Consequently, $(H, S, V, {\mathcal C})$ is a frame on
$T_0M$, which is called the \emph{Berwald frame}. 

\begin{lem} \label{Bcom_spray}
Consider $S$ a spray and let $(H, S, V, {\mathcal C} )$ be a fixed Berwald
  frame.
\begin{itemize}
\item[i)] The following formulae are satisfied:
\begin{align}
[{\mathcal C}, V]=0, \quad [S, H]=\nabla H + \rho V, \quad [S,V] = -H +
  \nabla V. \label{Berwald_s2} 
\end{align} 
\item[ii)] The rank of the $2$-form $\Omega$ defined by 
  \eqref{omegas} is given by 
\begin{align}
\operatorname{rank}(\Omega) =
 \left\{ 
        \begin{array}{cc}
4, & \quad \textrm{if} \ \alpha\left([H,V]\right) \neq 0; \\
2, & \quad \textrm{if} \ \alpha\left([H,V]\right) = 0.
        \end{array} 
\right. \label{ranko}
\end{align}
\end{itemize}
\end{lem}
\begin{proof}
i) The tangent structure $J$ is $0^{+}$-homogeneous, which means that
$[{\mathcal C}, J]=-J$. If we evaluate both sides of this formula on the
horizontal $2^{+}$-homogeneous vector field $H$ we obtain 
\begin{eqnarray*}
-V=-J(H)=[{\mathcal C}, J](H)=[{\mathcal C}, JH] - J[{\mathcal C}, H] =
  [{\mathcal C}, V] - V. \end{eqnarray*}
Thus formula \eqref{Berwald_s2} is true, which means
that $V$ is a $1^{+}$-homogeneous vector field.

Using formulae  \eqref{nabla} and \eqref{isotropic}, we have 
\begin{eqnarray*}
[S, H]=h[S, hH] + v[S, hH]= \nabla H + \Phi(H) = \nabla H + \rho V.
\end{eqnarray*}

From the properties of the dynamical covariant
derivative, \cite[Theorem 3.4]{BD09}, we have that $\nabla J=0$, and hence $J(\nabla
H)=\nabla J(H) = \nabla V$. From the second formula in \eqref{Berwald_s2} we obtain $J[S,H]
= \nabla V$. In order to prove the last formula in \eqref{Berwald_s2}, we use that
$H$ is horizontal and the definition \eqref{hv} of the horizontal
projector $h$.  We obtain
\begin{eqnarray*} 2H=\operatorname{Id}(H)+[J,S](H)=H + [J(H), S] -
  J[H, S]= H + [V,S] + \nabla V, \end{eqnarray*} which was to be shown. 

ii) We give a matrix representation of the $2$-form $\Omega$
with respect to the Berwald frame $(H, S, V, {\mathcal C})$. We
evaluate \eqref{omegas} on pairs of vector fields $X, Y \in \{H, S, V,
{\mathcal C}\}$. Since 
\begin{align*}
\frac{\alpha}{\rho}(S)=1, \ \frac{\alpha}{\rho}(X)=0, \forall X \in
  \{H, V, {\mathcal C}\}, 
\end{align*}
it follows that 
\begin{align*}
  d\left(\frac{\alpha}{\rho}\right)(X,Y) = - 
  \frac{\alpha}{\rho}\left([X, Y]\right),  \forall X,Y \in
  \{H, S, V, {\mathcal C}\}. 
\end{align*}
Using the commutation formulae \eqref{Berwald_s2}, we obtain the
following matrix representation of  $\Omega$ with respect
to $(H, S, V, {\mathcal C})$
\begin{align}
\Omega = \left(
\begin{array}{cccc}
0 & \alpha(\nabla H)/\rho &  -\alpha([H,V])/\rho & 0 \\
-\alpha(\nabla H)/\rho & 0 & 0 &1 \\
 \alpha([H,V])/\rho & 0 & 0 &0 \\
0 & -1 & 0 &0
\end{array} \label{omegab}
\right).\end{align}
From \eqref{omegab} we see that the rank
of $\Omega$ is given by formula \eqref{ranko}.
\end{proof}

\begin{defn} \label{regulars}
A spray $S$ is called \emph{regular} if the $2$-form $\Omega$ is an
almost symplectic form, which means that it satisfies the condition: $\operatorname{rank}(\Omega)=4$. The spray is called
\emph{degenerate} if $\operatorname{rank}(\Omega)=2$. 
\end{defn}

If we replace $H$ by $aH$ where
$a\in C^{\infty}(T_0M)$ is a nowhere vanishing $0^{+}$-homogeneous function, then
we have that $\alpha([aH, aV])=a^2\alpha([H,V])$ and hence the
regularity condition $\alpha([H,V])\neq 0$ does not depend on the
choice of the Berwald frame.

\subsection{Berwald frame for a Finsler function.} \label{BfF}

Consider the geodesic spray $S$ of  a Finsler function $F$ and consider the
Berwald frame $(H, S, V, {\mathcal C})$, determined the conditions
\eqref{Berwald_s1}, with the extra
condition that $H$ has the same length as $S$, which means that 
\begin{align}
G(H,H)=G(S,S)=F^2. \label{Berwald_f1}
\end{align}
The three normalised vector field $(H/F, S/F, V/F)$ were used in
\cite[\S 4.3]{BCS00} to provide an orthonormal frame for the
projective sphere bundle $SM$ of a $2$-dimensional Finsler metric.

Now, for the Berwald frame, uniquely associated to a Finsler function,
we will provide all the commutation formulae, extending the results of Lemma
\ref{Bcom_spray}.
\begin{lem}
\label{Bcom_finsler}
Consider $S$ the geodesic spray of a Finsler function $F$ and let
$( H, S, V, {\mathcal C})$ be its Berwald frame. 
\begin{itemize}
\item[i)] The Berwald frame satisfies the formulae
\begin{align}
\nonumber &  H(F^2)=0,\quad  V(F^2)=0, \quad \nabla H=0, \quad \nabla V=0, \\
& [S, H]= \rho V, \quad [S,V] = -H, \quad [H,V]=S+IH +S(I) V. 
\label{Berwald_f2} 
\end{align} 
\item[ii)] The geodesic spray is regular.
\end{itemize}
\end{lem}
\begin{proof}
First we prove the second part ii) of the lemma.
Using formulae \eqref{oomega}, \eqref{gomega} and the scaling condition
\eqref{Berwald_f1}  we obtain
\begin{align}
\Omega(H,V)=\frac{-1}{2F^2}\omega_{F^2}(H,V) =
  \frac{-1}{2F^2}\omega_{F^2}(H, -{\mathbb F} H) = \frac{-1}{2F^2} 2
  G(H,H) =-1. \label{omegahv1}
\end{align}
From the matrix representation \eqref{omegab} it is clear that
$\operatorname{rank}(\Omega)=4$, so the geodesic spray
$S$ is regular.

i) Equation \eqref{isomega} is equivalent to $d_hF^2=0$, which is
further equivalent to $hX(F^2)=0$, for all $X \in {\mathfrak
  X}(T_0M)$. It follows that the horizontal vector field $H$ also satisfies
$H(F^2)=0$.

The geodesic spray $S$ has scalar flag curvature, which
means that its Jacobi endomorphism is given by formula \eqref{sfc}. It follows
that the last condition \eqref{Berwald_s1}, for defining the horizontal
vector field $H$, can be written as follows
\begin{align}
0=2\alpha(H)=2\kappa Fd_JF(H)=\kappa d_JF^2(H)= \kappa JH(F^2) =
  \kappa V(F^2), \end{align} which means that $V(F^2)=0$, since
$\kappa \neq 0$. 

From the matrix representation \eqref{omegab} we obtain $\Omega(H,
{\mathcal C})=0$ and $\Omega(S, {\mathcal C})=1$.

Since the dynamical covariant derivative preserves the horizontal
distribution, it follows that $\nabla H=a_1S+a_2H$, for $a_1$, $a_2\in
C^{\infty}(T_0M)$, is a horizontal
vector field. It follows that $\nabla V =\nabla JH = J\nabla H =
a_1{\mathcal C} + a_2 V$.  Using the properties of the dynamical
covariant derivative that were proven in \cite{BD09} we have that $\nabla \Omega=0$ and $\nabla {\mathcal C}=0$,
and therefore
\begin{align*}
0 & =\left(\nabla \Omega\right)(H, {\mathcal C}) = \nabla \left(\Omega(H,
  {\mathcal C})\right) - \Omega(\nabla H,  {\mathcal C}) - \Omega(H,
  \nabla {\mathcal C}) = -\Omega
  (a_1S + a_2H, {\mathcal C}) = -a_1, \\
 0 & =\left(\nabla \Omega\right)(H, V) = \nabla \left(\Omega(H,
 V) \right) - \Omega(\nabla H, V) - \Omega(H,
  \nabla V) = -2 a_2 \Omega (H,V) = -2a_2
\end{align*}
and hence $\nabla H = \nabla V
=0$. Using these calculations and the last two commutation formulae
in \eqref{Berwald_s2}, it follows that the first two commutation
formulae in \eqref{Berwald_f2} are true.

Using the notation $[H,V]=b_1S+b_2H+b_3{\mathcal C}+b_4V$, for $b_1,
b_2$, $b_3, b_4\in C^{\infty}(T_0M)$, and the
regularity condition \eqref{omegahv1}, we have that $b_1=\alpha([H,V])/\rho=-\Omega(H,V)=1$.

From the first two commutation formulae
\eqref{Berwald_f2} and the Jacobi identity 
$[S,[H,V]]+[V,[S,H]]+[H,[V,S]]=0$,  we obtain  
\begin{align*}
0 & =[S, S+b_2H+b_3{\mathcal C}+b_4V] + [V, \rho V] + [H,H] \\ & =-b_3S +
  \left(S(b_2)-b_4\right) H +S(b_3){\mathcal C} + \left(S(b_4) +
  b_2\rho+ V(\rho)\right) V.
\end{align*}
Preceding calculations imply $b_3=0$, $b_4=S(b_2)$ and $S(b_4) +
b_2\rho + V(\rho)=0$. Using the notation $I=b_2$ we obtain that the
last commutation formula \eqref{Berwald_f2} is true and the
coefficient function $I$ satisfies $S^2(I)+I\rho + V(\rho) =0.$ 
\end{proof}

The three commutation formulae \eqref{Berwald_f2}, viewed as derivations on
$0^{+}$-homogeneous functions, were obtained first by Berwald
\cite[(7.6)]{Berwald41}. By comparing the last formula in \eqref{Berwald_f2}
and the first formula in \cite[(7.6)]{Berwald41} we obtain that the function
$I$ is the main scalar of the Finsler function.  
In the Riemannian case, the main scalar $I$ vanishes and the three
commutation formulae \eqref{Berwald_f2} reduce to the commutation
formulae \cite[Lemma 1]{Crampin14}. For a different derivation of the
Berwald frame with the commutation formulae \eqref{Berwald_f2} we
refer to \cite{VV01}, see also \cite[\S 9.9.1]{SLK14}, where the pull-back
formalism is applied.

\subsection{Berwald frame for a degenerate Finsler function.} \label{BfdF}

Consider a geodesic spray $S$ of a degenerate Finsler function
$F$; then $S$ is a solution of the equation
\eqref{isomega}. For the spray $S$, again, we consider a Berwald frame
determined by the conditions \eqref{Berwald_s1}.  First we extend the results of Lemma
\ref{Bcom_spray}.
\begin{lem}
\label{Bcom_degfinsler}
Consider a geodesic spray $S$ of a degenerate Finsler function $F$ and let
$(H, S, V, {\mathcal C}) $ be a Berwald frame. 
\begin{itemize}
\item[i)] The following formulae are valid:
\begin{align}
& H(F^2)=0, \quad V(F^2)=0, \quad \alpha \left(\nabla H \right)=0, \quad
  d_vF\left(\nabla V\right) =0, \nonumber \\
& \alpha\left([S,H]\right) = d_vF\left([S,H]\right) =0, \quad
  \alpha\left([S,V]\right) = d_vF\left([S,V]\right) =0, \label{Berwald_d1} \\
&  \alpha\left([H, V]\right) = d_vF\left([H, V]\right) =0. \nonumber
\end{align}
\item[ii)] The geodesic spray is degenerate.
\end{itemize}
\end{lem}
\begin{proof}
i) We obtain as above that
$d_hF^2=0$ and hence $H(F^2)=0$. The condition $\alpha(H)=0$ implies $V(F^2)=0$.

Since $F$ is a degenerate Finsler function,
$\operatorname{rank}(\omega_{F^2})=\operatorname{rank}(\Omega)=2$. Using
the matrix representation \eqref{omegab} of the $2$-form $\Omega$ with 
respect to the Berwald frame, we find that
$\Omega(H,S)=0$ and $\Omega(H,V)=0$. So, we have 
\begin{align*}
\alpha\left([H,V]\right) =-\rho \Omega(H,V)= 0, \quad
  \alpha\left(\nabla H\right) = \rho \Omega(H,
  S)=  0. \end{align*} 
Therefore, $\alpha(\nabla H)=0$ and hence $d_JF(\nabla
H)=0$. The dynamical covariant derivative preserves the horizontal
distribution, which is spanned by $H$ and $S$. The condition
$\alpha(\nabla H)=0$ means that the vector field $\nabla H$ does not
have a component along $S$ and hence $\nabla H=a_1 H$, for some
function $a_1\in C^{\infty}(T_0M)$. Using the fact that $\nabla J=0$,
it follows that $\nabla V =a_1 V$, which shows that the vertical
vector field $\nabla V$ does not have a component along ${\mathbb
  C}$. The last commutation formulae \eqref{Berwald_s2} can be written
now as follows:  
\begin{align}
[S,H]= a_1H+\rho V, \quad [S,V]=-H + a_1V.
\label{Berwald_d2}  \end{align} 
Commutation formulae \eqref{Berwald_d2} show that the vector fields
$[S,H]$ and $[S,V]$ have no components along $S$ and ${\mathcal C}$ and
therefore the corresponding formulae \eqref{Berwald_d1} are true.

We have seen already that the degeneracy condition on a Finsler function implies
$\alpha([H,V]) = 0$, which means that the vector field $[H,V]$ has no
component along $S$ and hence we can write it in the form 
$ [H, V] = a_2H + a_3 {\mathcal C} + a_4 V,$ for some functions  $a_2, a_3, a_4 \in
  C^{\infty}(T_0M).$  Using the Jacobi identity for the three vector fields $S, H, V$ and
the above expressions for the Lie brackets $[S, H]$, $[S, V]$ and
$[H,V]$ we obtain $a_3=0$ and therefore the last formulae
\eqref{Berwald_d1} are true.

ii) Since $F$ is a degenerate Finsler function, $\operatorname{rank}(g_{ij})=1$, which, in view of
\eqref{ogij} and \eqref{oomega}, is equivalent to
$\operatorname{rank}(\omega_{F^2})=\operatorname{rank}(\Omega)=2$,
thus, the spray $S$ is degenerate.  
\end{proof}

\section{Integrability of the Berwald distribution and Finsler metrizability.}
\label{Frobenius}

In this section we will prove that a $2$-dimensional spray $S$ is metrizable if and
only if the Berwald distribution \eqref{dhphi} is integrable. The
regularity of the corresponding Finsler function depends on the rank of
the $2$-form \eqref{omegas}. We treat separately
the regular and degenerate cases.

\subsection{Finsler metrizability.} \label{Fm}

For a $2$-dimensional, non-flat spray $S$, we consider the
Berwald distribution ${\mathcal D}$ \eqref{dhphi} and the $2$-form $\Omega$
\eqref{omegas}. The next theorem provides characterisations for the Finsler
metrizability of a regular spray, together with an algorithm that can
be used to construct effectively a Finsler function that metrizes the spray.  

\begin{thm} \label{thm1}
We consider a $2$-dimensional, non-flat spray $S$. The
following conditions are equivalent:
\begin{itemize}
\item[i)] $S$ is Finsler metrizable.
\item[ii)] $S$ is regular and the Berwald distribution ${\mathcal D}$ is integrable.
\item[iii)] There exists a closed $1$-form $\omega \in {\mathcal
    D}^{\ast}$ such that 
\begin{align} \operatorname{rank}\left( d_J\omega + 2
  \omega \wedge i_J\omega\right) =4. \label{rank4}\end{align} 
\end{itemize} 
\end{thm}
\begin{proof}
For the first implication $i) \Longrightarrow ii)$ we assume that $S$
is the geodesic spray of a Finsler function $F$. We consider the
Berwald frame associated to the Finsler function $F$ as it has been
described in Section \ref{BfF}. The conclusion comes from Lemma
\ref{Bcom_finsler}. The spray $S$ is regular and, due to the
commutation formulae \eqref{Berwald_f2}, we have that the distribution
${\mathcal D}=\operatorname{span}\{H, S, V\}$ is integrable, by the Frobenius Theorem.

For the second implication $ii) \Longrightarrow iii)$ we assume that
$S$ is a regular spray and the Berwald distribution ${\mathcal
  D}=\operatorname{span}\{H, S, V\}$ is integrable. Since 
$\operatorname{rank}{\mathcal D}=3$, it follows that
$\operatorname{rank} {\mathcal
    D}^{\ast}=1$. This freedom allows us to choose a $0^{+}$-homogeneous
  $1$-form $\omega\in {\mathcal D}^{\ast}$, which means that ${\mathcal L}_{{\mathcal C}}\omega =0$. We
  fix this $1$-form with the normalisation condition $i_{{\mathbb
      C}}\omega=1$. We will prove that the $1$-form $\omega$ satisfies
  the two conditions $iii)$.

Since the Berwald distribution ${\mathcal D}$ is integrable it follows
that for any two vector fields $X, Y\in {\mathcal D}$ we have
$[X,Y]\in {\mathcal D}$ and hence $\omega\left([X,Y]\right)=0$. Therefore
$d\omega(X,Y)=0$, for all $X,Y\in {\mathcal D}$. Using Cartan's
formula, as well as the defining properties of $\omega$, we have
\begin{align*} i_{{\mathcal C}}d\omega={\mathcal L}_{{\mathcal C}}\omega -
di_{{\mathcal C}}\omega =0. \end{align*}
It follows that $d\omega=0$. 

Since $\omega\in {\mathcal D}^{\ast}$ it follows that
$\omega \circ \Phi=0$. The spray $S$ is isotropic and its Jacobi
endomorphism $\Phi$ is given by formula \eqref{isotropic}, so we get
\begin{align*}
0 = \omega\circ \left(\rho J - \alpha \otimes {\mathcal C}\right) =
  \rho i_J\omega - \left(i_{{\mathcal C}}\omega \right) \alpha =  \rho i_J\omega -
  \alpha. \end{align*}
The above calculations imply 
\begin{align}
i_J\omega =\frac{\alpha}{\rho} \textrm{\ \ and \ therefore \ } \omega =
  i_{{\mathbb F}}\frac{\alpha}{\rho}. \label{ijomega} \end{align}
Using these two formulae we can see that the $2$-form
in formula \eqref{rank4} coincides with the $2$-form \eqref{omegas}. The
regularity condition of the spray $S$ implies that condition
\eqref{rank4} is satisfied.   

For the last implication $iii) \Longrightarrow i)$ consider a closed $1$-form
$\omega\in {\mathcal D}^{\ast}$ that satisfies the condition \eqref{rank4}. 

Again, the fact that $\omega\in {\mathcal D}^{\ast}$ implies 
$\omega \circ \Phi=0$, and as we have just seen formulae \eqref{ijomega} are true. Locally,
these two formulae can be written as follows
\begin{align*}
i_J\omega =\frac{\alpha_i dx^i}{\rho} \textrm{\ \ and \ therefore \ } \omega =
  \frac{\alpha_i \delta y^i}{\rho}. \end{align*}
From these two formulae we obtain $i_{{\mathcal C}}\omega =1$ and ${\mathcal L}_{{\mathbb
  C}}\omega=0.$

Since the $1$-form $\omega$ is closed, using Poincar\'e's Lemma we conclude that there exists a locally defined function $f$ on $T_0M$
such that $\omega=df$. We define $F=\exp(f)$ and we show that
$F$ is a Finsler function, whose geodesic spray is $S$.
Observe first that 
\begin{align*} 
{\mathcal C}(f)=i_{{\mathcal C}}df=i_{{\mathcal C}}\omega = i_{{\mathbb
  C}} i_{{\mathbb F}}\frac{\alpha}{\rho} = i_{{\mathbb F} {\mathbb
  C}}\frac{\alpha}{\rho}= i_S\frac{\alpha}{\rho} =1, \end{align*}
which implies that ${\mathcal C}(F)= {\mathcal C} \left(\exp(f)\right)=\exp(f)=F$, and,
therefore, $F$ is $1^{+}$-homogeneous.  

In view of the two formulae \eqref{ijomega} we have that the $2$-form
$\Omega$, \eqref{omegas}, is given by $\Omega =  d_J\omega + 2
 \omega \wedge  i_J\omega$. Now, using the assumption \eqref{rank4}
  we obtain $\operatorname{rank}(\Omega)=4$. First formula \eqref{ijomega} can be written now as follows
\begin{align*}
\frac{\alpha}{\rho}= i_J\omega = d_Jf =\frac{d_JF}{F} =\frac{d_JF^2}{2F^2}. 
\end{align*}
Using this formula, we see that the $2$-form $\Omega$ and the
function $F$ are related by formula \eqref{oomega}. It follows that
$\operatorname{rank}(\omega_{F^2})=4$, which means that $F^2$ is
regular and hence $F$ is indeed a Finsler function.

It remains to check that $S$ is the geodesic spray of the Finsler
function, which we show by proving that $d_hF=0$. We use
again the condition $\omega\in {\mathcal D}^{\ast} $, which implies
$\omega \circ h=0$. Since $\omega=df=\frac{dF}{F}$ it follows that
$d_hf=0$ and hence $d_hF=0$, which completes the proof.
\end{proof}

Criterion $iii)$ in Theorem \ref{thm1} shows ``where to look'' for
a Finsler function in the case when a spray $S$ is metrizable. It is
enough to pick a closed $1$-form $\omega$ from the $1$-dimensional
annihilator of the Berwald distribution ${\mathcal D}$. The two
conditions ${\mathcal L}_{{\mathcal C}}\omega=0$ and $i_{{\mathbb
    C}}\omega=1$ show that the $1$-form $\omega$ is unique. Condition $d\omega=0$ is
equivalent to the integrability of the distribution ${\mathcal D}$ and provides the
Finsler function $F$ that metrizes the spray $S$, through the
condition $\omega=\frac{dF}{F}$. In this case the freedom we have for
choosing the Finsler function is $1$ and it agrees to the result in 
\cite{EM16} for non-flat isotropic spray.

Criterion $ii)$ in Theorem \ref{thm1} is more geometric and it
shows that in the integrable case, the Berwald distribution ${\mathcal
  D}$ is tangent to a hyper-surface in $TM$ that represents the
indicatrix of the Finsler function that metrizes the spray.   

\subsection{Degenerate Finsler metrizability} \label{dFm}

In this section, we pay attention to the case when a $2$-dimensional,
non-flat spray $S$ is metrizable by a degenerate Finsler
function. Similarly with the regular case, we will show in the next
theorem that all the information regarding the metrizability of a
spray are encoded again into the Berwald
distribution \eqref{dhphi} and the $2$-form \eqref{omegas}.

 \begin{thm} \label{thm2}
Let $S$ be a $2$-dimensional, non-flat spray. The
following conditions are equivalent:
\begin{itemize}
\item[i)] $S$ is metrizable by a degenerate Finsler function.
\item[ii)] The spray $S$ is degenerate and the Berwald distribution ${\mathcal D}$ is integrable.
\item[iii)] There exists a closed $1$-form $\omega \in {\mathcal
    D}^{\ast}$ such that 
\end{itemize} 
\begin{align} \operatorname{rank}\left( d_J\omega + 2
  \omega \wedge  i_J\omega\right) =2. \label{rank2}
\end{align} 
\end{thm}
\begin{proof}
We use similar techniques and ideas that were used in the proof of Theorem
\ref{thm1}. We will emphasise only the aspects related to the degeneracy
of the spray and the corresponding degenerate Finsler function.

For the first implication we make use of Lemma
\ref{Bcom_degfinsler}. According to the second item of this lemma, we
have that the spray $S$ is degenerate. From the commutation formulae
\eqref{Berwald_d1} we obtain 
\begin{align*} d_vF\left([S,H]\right)= d_vF\left([S,V]\right) =
  d_vF\left([H, V]\right) =0, \end{align*} which show that the vector
  fields $[S,H]$, $[S,V]$ and $[H,V]$ have no components along the
  Liouville vector field ${\mathcal C}$. Using Frobenius Theorem it
  follows that the Berwald distribution ${\mathcal
  D}=\operatorname{span} \{H, S, V\}$ is integrable. 

To prove the second implication $ii) \Longrightarrow iii)$, we assume that
$S$ is a degenerate spray and the Berwald distribution ${\mathcal
  D}=\operatorname{span}\{H, S, V\}$ is integrable. We can either
  follow the arguments of Theorem \ref{thm1} for the same
  implication, or we can just take 
\begin{align}
\omega=i_{{\mathbb F}}\frac{\alpha}{\rho}  = \frac{\alpha_i \delta
  y^i}{\rho} \in {\mathcal D}^{\ast}. \label{omega_ifar} \end{align} 
With this choice we have that $i_{{\mathcal C}}\omega=1$, ${\mathcal L}_{{\mathbb
    C}}\omega=0$, which implies $i_{{\mathcal C}}d\omega=0$. Using the
assumption that the Berwald distribution ${\mathcal D}$ is integrable it
follows that there exists $\theta\in \Lambda^1(T_0M)$ such that
$d\omega=\omega \wedge \theta$. We evaluate both sides of this formula
on the Liouville vector field ${\mathcal C}$ to obtain $0=i_{{\mathbb
    C}} d\omega = \theta - \left(i_{{\mathbb  C}}\theta\right)
\omega$, whence $\theta=\left(i_{{\mathcal C}}\theta \right)\omega$ and
therefore $d\omega=0$. 

Using the expression of the $1$-form $\omega$ it follows that the
$2$-form in formula \eqref{rank2} coincides with the $2$-form
\eqref{omegas}. The assumption that the spray $S$ is degenerate
implies that $\operatorname{rank}(\Omega)=2$ and hence formula
\eqref{rank2} is valid.

To prove the last implication, we consider a closed $1$-form $\omega \in {\mathcal
  D}^{\ast}$ that satisfies formula \eqref{rank2}. Condition $\omega \in {\mathcal
  D}^{\ast}$ implies that $\omega$ is actually given by formula
\eqref{omega_ifar}. Similarly as in the proof of Theorem \ref{thm1},
we have that there exists a locally defined function $f$ on $T_0M$
such that $\omega=df$. The function $F=\exp(f)$ is $1^{+}$-homogeneous and
satisfies $d_hF=0$. Formula \eqref{rank2} implies that
$\operatorname{rank}(\Omega)=2$, and in view of formula \eqref{oomega}
it follows that $\operatorname{rank}(\omega_{F^2})=2$. Therefore $F$
is a degenerate Finsler function that metrizes the given spray $S$. 
\end{proof}

One can provide an equivalent reformulation of the Finsler
metrizability criteria of Theorems \ref{thm1} and \ref{thm2} using the holonomy
distribution ${\mathcal D}_{\mathcal H}$ introduced in
\cite{Muzsnay06}, as follows. The spray $S$ is metrizable by a
(degenerate) Finsler function if and only if the Berwald distribution
${\mathcal D}$ is integrable, which is equivalent to ${\mathcal
  D}={\mathcal D}_{\mathcal H}$. In this case, the metrizability
freedom of the spray $S$, \cite[Proposition 4.9]{EM16}, is
$m_s=\operatorname{codim}{{\mathcal D}_{{\mathcal H}}}
  =\operatorname{codim}{\mathcal D}=1$. Therefore, if a non-flat spray
  $S$ is metrizable, then the corresponding (degenerate) Finsler function is unique,
  up to a multiplicative constant.

\section{Examples} \label{examples}

In this section we exemplify the main results of our work. We will use both
criteria $ii)$ and $iii)$ in Theorems \ref{thm1} and \ref{thm2} to
test the metrizability of the given spray and to show how one can find the corresponding
(degenerate) Finsler function in the integrable case. 

\subsection{Metrizable sprays}

\subsubsection{The regular case.} We will use the next example to show
how the algorithms described in the proof of Theorem \ref{thm1} can be applied to test the Finsler metrizability of a spray and, in the
affirmative case, to construct the Finsler function. 

On $M=\{(x^1, x^2)\in {\mathbb R}^2, x^2>0\}$ we consider the SODE
\begin{align}
\frac{d^2x^1}{dt^2} - \frac{2}{x^2} \frac{dx^1}{dt}\frac{dx^2}{dt} =
0, \quad \frac{d^2 x^2}{dt^2} +\frac{1}{x^2}\left(
  \left(\frac{dx^1}{dt}\right)^2 -  \left(\frac{dx^2}{dt}\right)^2
\right) = 0. \label{sodeP}
\end{align}
The components of the induced nonlinear connection are given by
\begin{align} 
N_1^1=-\frac{y^2}{x^2}, \quad N_2^1=-\frac{y^1}{x^2}, \quad
  N_1^2=\frac{y^1}{x^2}, \quad N_2^2=-\frac{y^2}{x^2}. \label{nijP}
\end{align}
The local components of the Jacobi endomorphism are
\begin{align*}
R^1_1= -\frac{(y^2)^2}{(x^2)^2}, \quad R^1_2 = R^2_1 =
\frac{y^1y^2}{(x^2)^2}, \quad R^2_2=
-\frac{(y^1)^2}{(x^2)^2}. \end{align*}
The spray $S$ is isotropic, the components of the semi-basic
$1$-form $\alpha$ and the Ricci scalar are
\begin{align*}
\alpha_1=\frac{R^2_2}{y^1}=-\frac{y^1}{(x^2)^2}, \quad
\alpha_2=\frac{R^1_1}{y^2}=-\frac{y^2}{(x^2)^2}, \quad \rho=R^1_1+R^2_2=-\frac{1}{(x^2)^2}\{(y^1)^2+(y^2)^2\}. \end{align*}
We test first the metrizability of the spray using criterion
$iii)$ of Theorem \ref{thm1}.
All the information about the Finsler metrizability of the spray are
encoded into the $1$-form
\begin{align*}
\omega=i_{\mathbb{F}}\frac{\alpha}{\rho}=\frac{\alpha_1}{\rho}\delta
  y^1+\frac{\alpha_2}{\rho}\delta y^2 =
  \frac{y^1dy^1+y^2dy^2}{(y^1)^2+(y^2)^2}-\frac{1}{x^2}dx^2. \end{align*}
We have that 
\begin{align*} 
\Omega=\frac{-1}{(y^1)^2+(y^2)^2}\left(dx^1\wedge dy^1 + dx^2 \wedge
  dy^2\right), \end{align*}
so $\operatorname{rank}(\Omega)=4$ and hence $S$ is a
regular spray. Moreover,  $d\omega=0$ and hence $\omega = df$, for
$f(x,y)=\frac{1}{2}\ln\left((y^1)^2+(y^2)^2\right) -
\ln(x^2)$.
 It follows that $S$ is metrizable by the Finsler function
\begin{align} F(x,y)=\exp(f(x,y)) =
  \frac{\sqrt{(y^1)^2+(y^2)^2}}{x^2}, \label{FP}\end{align}
which is the Poincar\'e metric on the half-plane $M$. 

Now, we will check again the metrizability of the SODE \eqref{sodeP} using the
second criterion of Theorem \ref{thm1}. Consider a vector field $H\in
{\mathfrak X}(T_0M)$ satisfying conditions
\eqref{Berwald_s1}. One can choose such a vector field to be
$H=-y^2\delta/{\delta x^1} + y^1\delta /{\delta x^2}$. Therefore, the
Berwald frame $(H, S, V=JH)$ generates the Berwald distribution
${\mathcal D}$. From the following Lie brackets, we can see directly
that this distribution is integrable:
\begin{align*}
[H,V]=S, \quad [S,V]=-H, \quad [S,H]=\rho V.
\end{align*}
We want to find now the integral manifold $IM$ to the Berwald distribution
${\mathcal D}=\operatorname{Im}(h)\oplus \operatorname{Im}(\Phi)$. We
will search for the manifold $IM$ using the fact that it contains all
horizontal curves and the curves tangent to the vertical vector field $V$.     

A vertical curve $\gamma_v(t)=(x^i, y^i(t))$ is tangent to the vector field
$V$ if and only if 
\begin{align*}
\frac{dy^1}{dt} = -y^2, \quad \frac{dy^2}{dt} = y^1.
\end{align*}
With the initial condition $\gamma_v(0)=(x^i, y^i)$, we obtain that the curve
$\gamma_v(t)=(x^1, x^2, y^1\cos t - y^2 \sin t, y^1 \sin t - y^2 \cos t)$
belongs to the family of hyper-surfaces
\begin{align}
(y^1)^2+(y^2)^2=f(x^1,x^2), \label{hsP}
\end{align}
for some arbitrary function $f$ on the base manifold $M$. 

We will restrict the family of hyper-surfaces \eqref{hsP} by requiring
them to contain also horizontal curves. A curve $\gamma_h(t)=(x^i(t), y^i(t))$ is horizontal if and only if $v(\dot{\gamma_h}(t))=0$ and
therefore it satisfies the system of second order ordinary
differential equations
\begin{align}
\frac{d^2x^i}{dt^2} + N^i_j\frac{dx^j}{dt}= 0. \label{horP}
\end{align}
For the nonlinear connection \eqref{nijP} the system \eqref{horP}
becomes
\begin{align*}
\frac{dy^1}{dt}  -\left(\frac{y^2}{x^2}\frac{dx^1}{dt} +
  \frac{y^1}{x^2}\frac{dx^2}{dt} \right) =0, \quad 
\frac{dy^2}{dt}  + \frac{y^1}{x^2}\frac{dx^1}{dt} -
  \frac{y^2}{x^2}\frac{dx^2}{dt} =0.  
\end{align*}
We multiply the first equation by $y^1$, the second equation by
$y^2$, we add them to obtain
\begin{align*}
y^1\frac{dy^1}{dt} + y^2\frac{dy^2}{dt} - \frac{1}{x^2} \left(
  (y^1)^2+(y^2)^2\right) \frac{dx^2}{dt} =0. \end{align*}
Last equation can be written as 
\begin{align}
\frac{d}{dt}\left((y^1)^2 + (y^2)^2\right)  - \frac{2}{x^2}\left((y^1)^2 + (y^2)^2\right) \frac{dx^2}{dt} =0. \label{intP}
\end{align}
We want the horizontal curves to belong to the family of
hyper-surfaces \eqref{hsP}. Therefore, if we substitute the equations
\eqref{hsP} into \eqref{intP} we obtain 
\begin{align*}
\frac{d}{dt}\left(f(x^1, x^2)\right) - \frac{2}{x^2} f(x^1,
  x^2)\frac{dx^2}{dt} =0, \end{align*} which implies $f(x^1,
x^2)=c(x^2)^2$, $c\in {\mathbb R}^{\ast}$. Therefore the integral manifold of the Berwald
distribution ${\mathcal D}$ is given by 
\begin{align*}
IM=\{(x^i, y^i) \in T_0M, \frac{1}{(x^2)^2} \left((y^1)^2 +
  (y^2)^2\right) =c \}, \end{align*}
which represents the indicatrix of the Poincar\'e metric \eqref{FP}.

\subsubsection{The degenerate case.} The next example is a degenerate
spray. We will test its metrizability and obtain the corresponding
degenerate Finsler function using the methods provided by the two
criteria of Theorem \ref{thm2}.

On $M={\mathbb R}^2$, consider the SODE
\begin{align}
\frac{d^2x^1}{dt^2}+\frac{x^2}{1+(x^2)^2}\frac{dx^1}{dt}\frac{dx^2}{dt}=0,\quad
  \frac{d^2x^2}{dt}=0. \label{soded}
\end{align}
The coefficients of the nonlinear connection are
\begin{align}
N_1^1=\frac{x^2y^2}{2(1+(x^2)^2)},\
  N_2^1=\frac{x^2y^1}{2(1+(x^2)^2)},\ N_1^2=N_2^2=0. \label{nijd}
\end{align}
The local components of the Jacobi endomorphism and the Ricci scalar
are 
\begin{align*} R_1^1=\frac{(y^2)^2[(x^2)^2-2]}{4[(x^2)^2+1]^2}, \quad
  R_2^2=0, \quad \rho=\frac{(y^2)^2[(x^2)^2-2]}{4[(x^2)^2+1]^2}.\end{align*}
The semi-basic 1-form $\alpha/\rho$ has the components
\begin{align*}
\frac{\alpha_1}{\rho}=\frac{R_2^2}{y^1\rho}=0, \quad
  \frac{\alpha_2}{\rho}=\frac{R_1^1}{y^2\rho}=\frac{1}{y^2}. \end{align*}
The information about the metrizability and the regularity of the
spray are encoded into the $1$-form
\begin{align*}
\omega=i_{\mathbb{F}}\frac{\alpha}{\rho}=\frac{\alpha_1}{\rho}\delta
  y^1+\frac{\alpha_2}{\rho}\delta y^2 =
  \frac{1}{y^2}dy^2. \end{align*}
It follows that the corresponding $2$-form \eqref{omegas} is given by
\begin{align*}
\Omega=\frac{-1}{(y^2)^2}dx^2\wedge dy^2, \end{align*}
so $\operatorname{rank}\left(\Omega\right)=2$ and hence the spray $S$ is degenerate.
We also have that $d\omega=0$, since $\omega = df$, for
$f(x,y)=\ln|y^2|$. Hence the degenerate Finsler function that
metrizes the given spray is $F(x,y)=\exp(f(x,y))=|y^2|$.

We want now to test the metrizability of the system \eqref{soded}
using the second criterion of Theorem \ref{thm2}. For the given
system, we construct a Berwald frame using the conditions
\eqref{Berwald_s1}. If we choose $H=-y^2\frac{\delta}{\delta x^1}$,
then the Lie brackets of the vector fields $S, H$ and $V=JH$ are
\begin{align*}
[H,V]=0, \quad [S,V]=-H+ \frac{y^2x^2}{2(1+(x^2)^2)} V, \quad 
[S,H]=\frac{y^2x^2}{2(1+(x^2)^2)}H+\rho V.
\end{align*}
It follows that the Berwald distribution is integrable and hence the
system \eqref{soded} is metrizable by a degenerate Finsler function. 
If we compute the integral manifold of the Berwald distribution, we
obtain that it is given by 
\begin{align*}
IM=\{(x^1, x^2, y^1, y^2)\in T_0M, y^2=c\}, \end{align*}  which
represents the indicatrix of the degenerate Finsler function $F(x,y)=|y^2|$.

\subsection{Non-metrizable sprays} 

The following is an example of a spray whose metrizability freedom is
$0$ and its has been proposed by Elgendi and Muzsnay in
\cite{EM16}. We will also show that the spray is not Finsler metrizable, using different
techniques provided by Theorems \ref{thm1} and \ref{thm2}.

 On $M=\{(x^1, x^2) \in {\mathbb R}^2, x^2>0\}$, we consider the
 spray 
\begin{align*} S=y^1\frac{\partial}{\partial
  x^1}+y^2\frac{\partial}{\partial x^2}-2\bigg(\varphi
  y^1+\frac{y^1y^2}{2x^2}\bigg)\frac{\partial}{\partial
  y^1}-2\bigg(\varphi
  y^2-\frac{(y^1)^2}{4}\bigg)\frac{\partial}{\partial y^2},\end{align*}  where we use the notation $\varphi=(x^2(y^1)^2+(y^2)^2)^{1/2}.$

The coefficients of  the nonlinear connection are
\begin{align*}
  N_1^1=\frac{y^2}{2x^2}+\varphi+\frac{x^2(y^1)^2}{\varphi},\
  N_2^1=\frac{y^1}{2x^2}+\frac{y^1y^2}{\varphi},\
  N_1^2=-\frac{y^1}{2}+\frac{x^2y^1y^2}{\varphi},\
  N_2^2=\varphi+\frac{(y^2)^2}{\varphi}, \end{align*}
while the local components of the corresponding Jacobi endomorphism
and Ricci scalar are
\begin{align*} R_1^1=\frac{(y^2)^2[4(x^2)^2+1]}{4(x^2)^2},\
  R_2^2=\frac{(y^1)^2[4(x^2)^2+1]}{4(x^2)^2},\
  \rho=((y^1)^2+(y^2)^2)\frac{[4(x^2)^2+1]}{4(x^2)^2}.\end{align*}
The spray $S$ is isotropic and the semi-basic 1-form $\alpha/\rho$ has the components
\begin{align*}
\frac{\alpha_1}{\rho}=\frac{y^1}{(y^1)^2+(y^2)^2}, \quad
  \frac{\alpha_2}{\rho}=\frac{y^2}{(y^1)^2+(y^2)^2}.
\end{align*}
Having the form for $\alpha/\rho$, we compute the $1$-form 
\begin{align*}
\omega=i_{{\mathbb F}}\frac{\alpha}{\rho}=\frac{\alpha_1}{\rho}\delta
  y^1+\frac{\alpha_2}{\rho}\delta y^2.
\end{align*}
In the above formula, we replace $\delta y^i=d y^i+N_j^idx^j$, and by
a direct calculation we find that $d\omega\neq 0$. Therefore, the two Theorems
\ref{thm1} and \ref{thm2} tell us that  our spray $S$ is not Finsler metrizable. 

We can reach the same conclusion by using the Berwald frame. 
For the given spray $S$, we search for a horizontal vector field
that satisfies equations \eqref{Berwald_s1}. In other words we
search for $H=H^1{\delta}/{\delta x^1} + H^2{\delta}/{\delta x^2}$
that satisfies $\alpha(H)=0$ and  $[\mathbb{C},H]=H$.
In order to check the second criterion of Theorems \ref{thm1} and
\ref{thm2}, we can take any solution of the above system. Such a
solution is given by $H^1=-y^2$ and $H^2=y^1.$ We consider the
vertical vector field $V=JH$. According to last
formula in \eqref{Berwald_s2} we have that $[S,V]=-H+\nabla V$. If we compute
the dynamical covariant derivative $\nabla V$ we obtain that it has a
non-vanishing component along the Liouville vector field ${\mathcal
  C}$. Therefore, the Berwald distribution ${\mathcal D}=\operatorname{span}\{H, S, V\}$ is not integrable and hence the spray $S$ is
not Finsler metrizable. The same conclusion follows using the metrizability
criterion introduced by Muzsnay in \cite{Muzsnay06}, since from the
above Lie bracket we have that ${\mathcal C}\in {\mathcal D}_{\mathcal H}$.

\subsection*{Acknowledgments} We express our thanks to Professor J\'ozsef Szilasi for the careful
reading of the manuscript and the suggestions that led to this
version.  E.T has been supported by ``Erasmus+ Mobility Project for Higher Education
Students and Staff with Partner Countries'' and thanks Professor Nabil
L. Youssef for his continuous encouragement.

\end{document}